\documentclass[11pt,a4paper]{amsart}
\usepackage[utf8]{inputenc}
\usepackage{amsmath}
\usepackage{amsfonts}
\usepackage{amssymb}
\usepackage{graphicx}
\usepackage{booktabs}
\usepackage{nicefrac}
\usepackage{tikz}
\usepackage[ruled,vlined]{algorithm2e}

\usepackage[left=2.2cm,right=2.2cm,top=2cm,bottom=2cm]{geometry}

\usepackage{tikz}
\usetikzlibrary{decorations.pathreplacing}
\usetikzlibrary{fadings}
\newcommand{\D}{{\mathop{}\!\mathrm{d}}} 
\newcommand{\R}{\mathbb{R}}

\newcommand{\N}{\mathbb{N}}

\newcommand{\PP}{\mathbb{P}}

\newcommand{\E}{\mathbb{E}}

\newcommand{\T }{\mathcal{T}}
\newcommand{\A}{\mathcal{A}}
\newcommand{\X}{\mathcal{X}}
\newcommand{\one}{ 1 \hspace{-3pt} \mathrm{l}} %

\newcommand{\ab}{{\mathbf{a}}}
\numberwithin{equation}{section}  
\newtheorem{defn}{Definition}[section]
\newtheorem{exa}[defn]{Example}
\newtheorem{rem}[defn]{Remark}
\newtheorem{thm}[defn]{Theorem}

\newtheorem{lem}[defn]{Lemma}

\newtheorem{s_asu}[defn]{Standing Assumption}

\usepackage{hyperref}
\usepackage{xcolor}
\hypersetup{
    colorlinks,
    linkcolor={red!50!black},
    citecolor={blue!50!black},
    urlcolor={blue!80!black}
}

\title[Bounding the Values of Markov Decision Problems]{Bounding the Difference between the Values of Robust and Non-Robust Markov Decision Problems}

\author[A. Neufeld, J. Sester]{ Ariel Neufeld$^{1}$, Julian Sester$^{2}$}

\begin{document}

\vspace*{-0.7cm}
\maketitle

\begin{center}
\normalsize{\today} \\ \vspace{0.5cm}
\small\textit{$^{1}$NTU Singapore, Division of Mathematical Sciences,\\ 21 Nanyang Link, Singapore 637371.\\
$^{2}$National University of Singapore, Department of Mathematics,\\ 21 Lower Kent Ridge Road, 119077.}                                                                                                                              
\end{center}

\begin{abstract}~
In this note we provide an upper bound for the difference between the value function of a distributionally robust Markov decision problem and the value function of a non-robust Markov decision problem, where the ambiguity set of probability kernels of the   distributionally robust Markov decision process is described by a Wasserstein-ball around some reference kernel whereas the non-robust Markov decision process behaves according to a fixed probability kernel contained in the ambiguity set. Our derived upper bound for the difference between the value functions is dimension-free and depends linearly on the radius of the Wasserstein-ball.
 \\
 \\
\textbf{Keywords: }{Markov Decision Process, Wasserstein Uncertainty, Distributionally Robust Optimization, Reinforcement Learning}
\end{abstract}

\section{Introduction}
Markov decision processes enable to model non-deterministic interactions between an agent and its environment within a tractable stochastic framework. At each time $t$ an agent observes the current state and takes an action which leads to an immediate reward. The goal of the agent then is to optimize its expected cumulative reward. Mathematically, Markov decision problems are solved based on a dynamic programming principle, whose framework builds the fundament of many reinforcement learning algorithms such as, e.g., the $Q$-learning algorithm. We refer to \cite{bauerle2011markov}, \cite{feinberg2012handbook}, \cite{puterman1990markov}, \cite{puterman2014markov} for the theory of Markov decision processes and to \cite{angiuli2021reinforcement}, 
\cite{cao2021deep}, \cite{charpentier2021reinforcement},
\cite{kallus2020double},
\cite{kaelbling1996reinforcement},
\cite{levin1998using},
\cite{natarajan2022planning},
\cite{sutton2018reinforcement},
\cite{white1993survey}  for their applications, especially in the field of reinforcement learning. 

In the classical setup for Markov decision problems, the transition kernel describing the transition probabilities of the underlying Markov decision processes is given. Economically, this means that the agent possesses the knowledge of the true distribution of the underlying process, an assumption which typically cannot be justified in practice. To address this issue, academics have recently introduced a robust version of the Markov decision problem accounting for a possible misspecification of the assumed underlying probability kernel that describes the dynamics of the state process. Typically, one assumes that the agent possesses a good guess of the true but to the agent unknown probability kernel, but due to her uncertainty decides to consider the worst case among all laws which lie within a ball of certain radius around the estimated probability kernel with respect to some distance, e.g. the Wasserstein distance or the KL-distance. We refer to \cite{bauerle2021distributionally_2}, \cite{bauerle2021q}, \cite{chen2019distributionally},
\cite{ElGhaouiNilim2005robust},
\cite{li2023policy}, \cite{liu2022distributionally},
\cite{mannor2016robust}, \cite{neufeld2022robust}, \cite{neufeld2023markov},  \cite{panaganti2022sample}, \cite{si2020distributional}, \cite{si2020distributionally},  \cite{uugurlu2018robust}, \cite{wang2022policy}, \cite{wiesemann2013robust}, \cite{wiesemann2014distributionally}, \cite{xu2010distributionally}, \cite{yang2021towards}, and \cite{zhou2021finite} for robust Markov decision problems and corresponding reinforcement learning based algorithms to solve them.

In this note, the goal is to analyze the difference between the value function of the corresponding Markov decision problem with respect to the true (but to the agent unknown) probability kernel and the one of the robust Markov decision problem defined with respect to some Wasserstein-ball around the by the agent estimated transition kernel. Note that the estimated transition kernel does not necessarily need to coincide with the true probability kernel, however we assume that the agent’s guess is good enough that the true probability kernel lies within the Wasserstein-ball around the estimated probability kernel.

Similar, while not identical, research questions have been studied in
\cite{bartl2022sensitivity}, \cite{kiszka2022stability}, \cite{kern2020sensitivity}, \cite{muller1997does}, \cite{zahle2022concept} mainly focussing on establishing stability results for value functions w.r.t.\,the choice of the underlying transition probability. In \cite[Theorem 4.2]{muller1997does}, the author presents a state-dependent bound on the difference between iterations of value functions (obtained via the so called value iteration algorithm) of two Markov decision processes implying that these iterations depend continuously on the transition kernels. As a refinement of \cite[Theorem 4.2]{muller1997does} and also of the related result obtained in  \cite[Theorem 2.2.8]{kern2020sensitivity}, the result from \cite[Theorem 6.2]{zahle2022concept} shows that in a finite time horizon setting the difference between the value functions of two Markov decision processes with different transition probabilities can be bounded by an expression depending on a certain tailored distance between the transition probabilities. In \cite{kiszka2022stability} the author proposes a semi-metric for Markov processes which allows to determine bounds for certain types of linear stochastic optimization problems, compare  \cite[Theorem 3]{kiszka2022stability}.  The authors from \cite{{bartl2022sensitivity}} study the sensitvity of multi-period stochastic optimization problems over a finite time horizon w.r.t.\,the underlying probability distribution in the so called adapted Wasserstein distance. They show in \cite[Theorem 2.4]{bartl2022sensitivity}  that the value function of their robust optimization problem with corresponding ambiguity set being a Wasserstein-ball around a reference measure can be approximated by the corresponding value function of the non-robust optimization problem defined with respect to the reference measure plus an explicit correction term. Vaguely speaking (as the optimization problem in \cite{bartl2022sensitivity} is technically speaking not comparable to our setting), this is similar to our analysis in the special case where our reference measure coincides with the true measure.

Under some mild assumptions, we obtain in Theorem~\ref{thm_main_result} an \textit{explicit} upper bound for the difference between the value function of the robust and the non-robust Markov decision problem which only depends on the radius $\varepsilon$ of the Wasserstein-ball, the discount factor $\alpha$, and the Lipschitz constants of the reward function and the true transition kernel. In particular, we obtain that the difference of the two value functions only grows \textit{at most linearly in the radius} $\varepsilon$ and \textit{does not depend on the dimensions of the underlying state and action space}.

The remainder of this note is as follows. In Section~\ref{sec_setting} we introduce the underlying setting which is used to derive our main result reported in Section~\ref{sec_main}. The proof of the main result and auxiliary results necessary for the proof are reported in Section~\ref{sec_proofs}.

\section{Setting}\label{sec_setting}

We first present the underlying setting to define both robust and non-robust Markov decision processes which we then use to compare their respective value functions. 
\subsection{Setting}\label{subsec_setting}
As state space we consider a closed subset $\X \subseteq\R^d$ for some $d\in \N$, equipped with its Borel $\sigma$-field $\mathcal{F}_{\X}$, which we use to define the infinite Cartesian product
\[
\Omega:=\X^{\N_0}=\X\times \X \times \cdots
\]
and the $\sigma$-field $\mathcal{F}:={\mathcal{F}_{\X}}\otimes \mathcal{F}_{\X} \otimes \cdots$.
For any $q\in \N$, we denote by $\mathcal{M}_1^q(\X)$ the set of probability measures on $\X$ with finite $q$-moments and write $\mathcal{M}_1(\X):=\mathcal{M}_1^1(\X)$ for brevity.
We define on $\Omega$ the  infinite horizon stochastic process $\left(X_{t}\right)_{t\in \N_0}$  via the canonical process $X_t(\left(\omega_0,\omega_1,\dots,\omega_t,\dots\right)):=\omega_t$ for $(\omega_0,\omega_1,\dots,\omega_t,\dots) \in \Omega$,  $t \in \N_0$.

To define the set of controls (also called actions) we fix a compact set $A \subseteq \R^m$ for some $m \in \N$, and set
\begin{align*}
\mathcal{A}:&=\left\{\ab=(a_t)_{t \in \N_0}~\middle|~(a_t)_{t \in \N_0}: \Omega \rightarrow A;~a_t \text{ is } \sigma(X_{t})\text{-measurable} \text{ for all } t \in \N_0 \right\}\\
&=\left\{\left(a_t(X_t)\right)_{t\in \N_0}~\middle|~ a_t:\X \rightarrow A \text{ Borel measurable for all } t \in \N_0 \right\}.
\end{align*}

Next, we define the $q$-Wasserstein-distance $d_{W_q}(\cdot,\cdot)$ for some  $q\in \N$. For any $\PP_1,\PP_2 \in \mathcal{M}_1^q(\X)$ let $d_{W_q}(\PP_1,\PP_2)$ be defined as 
\[
d_{W_q}(\PP_1,\PP_2):=\left(\inf_{\pi \in \Pi(\PP_1,\PP_2)}\int_{\X \times \X} \|x-y\|^q \D \pi(x,y)\right)^{1/q},
\]
where $\|\cdot\|$ denotes the Euclidean norm on $\R^d$, and where $\Pi(\PP_1,\PP_2)$ denotes the set of joint distributions of $\PP_1$ and $\PP_2$. Moreover, we denote by $\tau_q$ the Wasserstein $q$ - topology induced by the convergence w.r.t.\,$d_{W_q}$.


To define an ambiguity set of probability kernels, we first fix  throughout this paper some  $q\in \N$ and $\varepsilon>0$. Then, we define as ambiguity set of probability kernels
\begin{equation}\label{eq_definition_wasserstein_ball}
\X \times A  \ni (x,a) \twoheadrightarrow  \mathcal{P}(x,a):=\mathcal{B}^{(q)}_\varepsilon\left(\widehat{\PP}(x,a)\right):=\left\{\PP\in \mathcal{M}_1(\X)~\middle|~d_{W_q}(\PP,\widehat{\PP}(x,a)) \leq  \varepsilon \right\}
\end{equation}
with respect to some center 
$$\X \times A \ni (x,a) \mapsto \widehat{\PP}(x,a) \in  (\mathcal{M}^q_1(\X),\tau_q),$$ 
meaning that $\mathcal{B}^{(q)}_\varepsilon\left(\widehat{\PP}(x,a)\right)$ denotes the $q$-Wasserstein-ball (also called Wasserstein-ball of order $q$) with $\varepsilon$-radius and center $\widehat{\PP}(x,a)$.

Under these assumptions we define for every $x \in \X, \ab \in \mathcal{A}$ the set of admissible measures on $(\Omega, \mathcal{F})$ by
\begin{align*}
\mathfrak{P}_{x,\ab}:=\bigg\{\delta_x \otimes \PP_0\otimes \PP_1 \otimes \cdots~\bigg|~&\text{ for all } t \in \N_0:~\PP_t:\X \rightarrow \mathcal{M}_1(\X) \text{ Borel-measurable, } \\ 
&\text{ and }\PP_t(\omega_t) \in \mathcal{P}\left(\omega_t,a_t(\omega_t)\right)\text{ for all } \omega_t\in \X \bigg\},
\end{align*}
where the notation $\PP=\delta_x \otimes\PP_0\otimes \PP_1 \otimes\cdots \in \mathfrak{P}_{x,\ab}$ abbreviates
\[
\PP(B):=\int_{\X}\cdots \int_{\X} \cdots \one_{B}\left((\omega_t)_{t\in \N_0}\right) \cdots \PP_{t-1}(\omega_{t-1};\D\omega_t)\cdots \PP_0(\omega_0;\D\omega_1) \delta_x(\D \omega_0),\qquad B \in \mathcal{F}.
\]

\subsection{Problem Formulation and Standing Assumptions}
Let $r:\X \times A \times \X \rightarrow \R$ be some \emph{reward function}. We assume from now on that it fulfils the following assumptions.
\begin{s_asu}[Assumptions on the reward function]\label{asu_2}~
\\
The reward function $r:\X \times A \times \X \rightarrow \R$ satisfies the following.
\begin{itemize}
\item[(i)] The map 
\[
\X \times A \times \X \ni (x_0,a,x_1) \mapsto r(x_0,a,x_1)\in \R
\]
is Lipschitz continuous with constant $L_r>0$.
\item[(ii)] If $\X$ is unbounded and $q\in \N$ defined in \eqref{eq_definition_wasserstein_ball} satisfies $q=1$, then we additionally assume that
\begin{equation*}\label{eq_c_bounded}
\sup_{x_0,x_1\in \X, a\in A} |r(x_0,a,x_1)|<\infty.
\end{equation*}
\end{itemize}
\end{s_asu}
Note that Assumption~\ref{asu_2}(i) implies that the reward $r$ is bounded whenenver $\X$ is bounded. 
Next, we impose the following standing assumption on our reference probability kernel modeled by the center of the $q$-Wasserstein ball.
\begin{s_asu}[Assumption on the center of the ambiguity set]\label{asu_p}~
	Let $q\in \N$ be defined in \eqref{eq_definition_wasserstein_ball}. Then the center $\X \times A \ni (x,a) \mapsto \widehat{\PP}(x,a) \in  (\mathcal{M}^q_1(\X),\tau_q)$ satisfies the following.
	\begin{itemize}
		\item [(i)] The map $\X \times A \ni (x,a) \mapsto \widehat{\PP}(x,a) \in  (\mathcal{M}^q_1(\X),\tau_q)$ is continuous.
		\item [(i+)] If 
		the reward function $r$ is unbounded, 
		then we assume instead of $(i)$ the stronger assumption that $\widehat{\PP}$ is Lipschitz continuous, i.e.\ that there exists $L_{\widehat{\PP}}>0$ such that 
		\begin{equation*}
		d_{W_q}(\widehat{\PP}(x,a),\widehat{\PP}(x',a'))\leq L_{\widehat{\PP}}\big(\Vert x-x' \Vert + \Vert a-a' \Vert\big) \text{ for all } x,x' \in \X,~a,a'\in A.	
		\end{equation*}
	\end{itemize}
\end{s_asu}
Finally, we assume the following on the discount factor $\alpha\in (0,1)$.
\begin{s_asu}[Assumption on the discount factor]\label{asu_alpha}
	Let $q\in \N$, $\varepsilon>0$ be defined in \eqref{eq_definition_wasserstein_ball} and $\X \times A \ni (x,a) \mapsto \widehat{\PP}(x,a) \in  (\mathcal{M}^q_1(\X),\tau_q)$ be defined in Assumption~\ref{asu_p}. Then the discount factor $\alpha$ satisfies $0<\alpha<\frac{1}{C_P}$, where $1\leq C_P<\infty$ is defined by
\begin{align*}
C_P=	\begin{cases}
\max\Bigg\{	1+	\varepsilon
+
\sup\limits_{a\in A}\inf\limits_{x\in \X}\bigg\{
\displaystyle\int\limits_{\X }\|z\|\, \widehat{\PP}(x,a) (\D z) 
+
L_{\widehat{\PP}}\|x\|
\bigg\}, L_{\widehat{\PP}}\Bigg\}  & \mbox{if the reward $r$ is unbounded;}\\
		1 & \mbox{else.}	
	\end{cases}
\end{align*}
\end{s_asu}

Our goal is to compare the \emph{value} of the robust Markov decision problem with the  \emph{value} of the non-robust Markov decision problem. To define the robust value function, for every initial value $x\in \X$, one maximizes the expected value of $\sum_{t=0}^\infty \alpha^tr(X_{t},a_t,X_{t+1})$ under the worst case measure from $\mathfrak{P}_{x,\ab}$ over all possible actions $\ab \in \A$. More precisely, we introduce the robust value function by
\begin{equation}\label{eq_robust_problem_1}
\begin{aligned}
   \X \ni x \mapsto V(x):=\sup_{\ab \in \mathcal{A}}\inf_{\PP \in \mathfrak{P}_{x,\ab}} \left(\E_{\PP}\bigg[\sum_{t=0}^\infty \alpha^tr(X_{t},a_t,X_{t+1})\bigg]\right).
\end{aligned}
\end{equation}
To define the non-robust value function under the true but to the agent unknown probability kernel $\PP^{\operatorname{true}}$ contained in the ambiguity set $\mathcal{P}$, we impose the following assumptions on $\PP^{\operatorname{true}}$.
\begin{s_asu}[Assumptions on the true probability kernel]\label{asu_3}
Let $q\in \N$ be defined in \eqref{eq_definition_wasserstein_ball}.
Then the true (but unknown) probability kernel $\X \times A \ni(x,a) \mapsto \PP^{\operatorname{true}}(x,a) \in  (\mathcal{M}^q_1(\X),\tau_q)$ satisfies the following.
\begin{itemize}
	\item[(i)] We have that $\PP^{\operatorname{true}}(x,a) \in \mathcal{P}(x,a)$ for all $(x,a)\in \X \times A$.
	\item[(ii)]	$\PP^{\operatorname{true}}$
	is $L_P$-Lipschitz with constant 
\begin{equation}\label{eq_gamma_Lp_1}
0\leq  L_P < \tfrac{1}{\alpha},
\end{equation}
where $0<\alpha<1$ is defined in Assumption~\ref{asu_alpha}, i.e., we have
\begin{equation}\label{eq_Lp_Lipschitz}
d_{W_q}\left(\PP^{\operatorname{true}}(x,a),\PP^{\operatorname{true}}(x',a')\right) \leq L_P\left(\|x-x'\|+\|a-a'\|\right) \text{ for all } x,x' \in \X,~a,a'\in A.
\end{equation}
\end{itemize}
\end{s_asu}
Then, we introduce the non-robust value function under the true (but to the agent unknown) transition kernel by
\begin{equation}\label{eq_non_robust_problem_1}
\begin{aligned}
   \X \ni x \mapsto V^{\operatorname{true}}(x):=\sup_{\ab \in \mathcal{A}} \left(\E_{{\PP}^{\operatorname{true}}_{x,\ab}}\bigg[\sum_{t=0}^\infty \alpha^tr(X_{t},a_t,X_{t+1})\bigg]\right),
\end{aligned}
\end{equation}
where we denote for any $x \in \X$, $\ab \in \mathcal{A}$, $$\PP_{x,\ab}^{\operatorname{true}}:=\delta_{x} \otimes \PP^{\operatorname{true}} \otimes \PP^{\operatorname{true}} \otimes \PP^{\operatorname{true}} \otimes \PP^{\operatorname{true}}  \cdots \in \mathcal{M}_1(\Omega).$$
Note that Assumptions~\ref{asu_2}--\ref{asu_3} ensures that the dynamic programming principle holds for both the robust and non-robust Markov decision problem, see \cite[Theorem 2.7]{neufeld2023markov}.

\section{Main Result}\label{sec_main}
As a main result we establish a bound on the difference between the value function of the Markov decision process with fixed reference measure defined in \eqref{eq_non_robust_problem_1}, and the value function of the robust Markov decision process defined in \eqref{eq_robust_problem_1}.
\begin{thm}\label{thm_main_result}
Let all Assumptions~\ref{asu_2}--~\ref{asu_3} hold true.
\begin{itemize}
\item[(i)]
Then, for any $x_0\in \X$ we have 
\begin{equation}\label{eq_bound_main_thm}
0 \leq V^{\operatorname{true}}(x_0)-V(x_0) \leq 2 L_r \varepsilon \left(1+\alpha\right)\sum_{i=0}^\infty \alpha^i \sum_{j=0}^i(L_P)^j < \infty.
\end{equation}
\item[(ii)]
Moreover, in the special case that $\PP^{\operatorname{true}} = \widehat{\PP}$, we obtain for any $x_0\in \X$ that
\begin{equation}\label{eq_bound_main_thm_1}
0 \leq V^{\operatorname{true}}(x_0)-V(x_0) \leq  L_r \varepsilon \left(1+\alpha\right)\sum_{i=0}^\infty \alpha^i \sum_{j=0}^i(L_P)^j < \infty.
\end{equation}
 \end{itemize}
\end{thm}
We highlight that the upper bound from \eqref{eq_bound_main_thm} depends only on $\varepsilon, \alpha$, and the Lipschitz-constants $L_r$ and $L_P$. In particular, the upper bound depends linearly on the radius $\varepsilon$ of the Wasserstein-ball and is independent of the current state $x_0$ and the dimensions $d$ and $m$ of the state and action space, respectively.

\begin{rem}The assertion from Theorem~\ref{thm_main_result} also carries over to the case of autocorrelated time series where one assumes that the past $h \in \N \cap [2,\infty)$  values of a time series $(Y_{t})_{t \in \{-h,\dots-1,0,1\dots\}}$ taking values in some closed subset $\mathcal{Y}$ of $\R^{D}$ for some $D\in \N$ may have an influence on the next value. This can be modeled by defining the state process $X_t:=(Y_{t-h+1},\dots,Y_t) \in \mathcal{Y}^h=:\mathcal{X} $, $t\in \N_0$. In this setting, the subsequent state $ X_{t+1}=(Y_{t-h+2},\dots,Y_{t+1}) $ shares $h-1$ components with the preceding state $X_t=(Y_{t-h+1},\dots,Y_t)$ and uncertainty is only inherent in the last component $Y_{t+1}$. Thus, we consider a reference kernel of the form $\X \times A \ni (x,a) \mapsto \PP^{\operatorname{true}}(x,a)  = \delta_{\pi(x)} \otimes \widetilde{\PP}^{\operatorname{true}}(x,a) \in \mathcal{M}_1(\X)$, where $\widetilde{\PP}^{\operatorname{true}}(x,a) \in \mathcal{M}_1(\mathcal{Y})$ and $\X \ni (x_1,\dots,x_h) \mapsto \pi(x):= (x_2,\dots,x_h)$ denotes the projection on the last $h-1$ components.
In this setting, for $q\in \N$ and  $\varepsilon>0$, the ambiguity set is given by
\begin{equation*}
\begin{aligned}
\X \times A \ni (x,a) \twoheadrightarrow \mathcal{P}(x,a):=\Bigg\{&\PP \in \mathcal{M}_1(\mathcal{X})~\text{ s.t. }~\\
&\PP= \delta_{\pi (x)} \otimes \widetilde{\PP}\text{ for some } \widetilde{\PP} \in \mathcal{M}_1(\mathcal{Y}) \text{ with} ~W_q(\widetilde{\PP},\widetilde{\PP}^{\operatorname{true}}(x,a))\leq \varepsilon \Bigg\}.
\end{aligned}
\end{equation*}
The described setting is discussed in more detail in \cite[Section 3.3]{neufeld2023markov} or \cite[Section 2.2]{neufeld2022robust}. Typical applications can be found in finance and include portfolio optimization, compare \cite[Section 4]{neufeld2023markov}.
\end{rem}

\begin{exa}[Coin Toss]\label{exa_toin_coss}
To illustrate the applicability of Theorem 3.1, we study an example similar to the one provided in \cite[Example 4.1]{neufeld2022robust}. To this end, we consider an agent who at each time tosses $10$ coins and observes the number of heads. Thus, we  model the environment by a state space $\mathcal{X}:=\{0,\dots,10\}$. Prior to the toss, the agent can bet whether in the next toss of $10$ coins the sum of heads will be smaller ($a=-1$) or larger $(a=1)$ than the previous toss. She gains $1\$$ if the bet is correct and in turn has to pay $1\$$ if it is not (without being rewarded/punished if the sum of heads remains the same). Moreover, the agent can also
decide not to bet for the toss (by choosing $a=0$). We model this via the reward function
\[
\X \times A \times \X \ni (x,a,x') \mapsto r(x,a,x'):=a\one_{\{x<x'\}}-a\one_{\{x>x'\}},
\]
where the possible actions are given by $ A:=\{-1,0,1\}$.
The reference measure in this setting assumes a fair coin, and therefore (independent of the state action pair) is a  binomial distribution with $n=10,p=0.5$, i.e., 
$$\X \times A \ni (x,a) \mapsto \PP^{\operatorname{true}}(x,a)=\widehat{\PP}(x,a):= \operatorname{Bin}(10,0.5).$$
In the described setting it is easy to see that $r$ is Lipschitz-continuous with Lipschitz constant
$$
L_r = \left(\max_{y_0,y_0',x_1,x_1'\in \mathcal{X},~b,b'\in A \atop (y_0,b,x_1) \neq (y_0',b',x_1')} \frac{|r(y_0,b,x_1)-r(y_0',b',x_1')|}{\|y_0-y_0'\|+\|b-b'\|+\|x_1-x_1'\|}\right) =1.
$$
Moreover, we have $L_P=0$.
In Figure~\ref{fig_coin_toss} we plot the corresponding upper bound from \eqref{eq_bound_main_thm_1} against the difference  $V^{\operatorname{true}}(x_0)-V(x_0)$ for different initial values $x_0$ and different levels of $\varepsilon$ used for the computation of $V$ with $\alpha =0.45$. The value functions are computed using the robust $Q$-learning algorithm proposed in \cite{neufeld2022robust}. The used code can be found under \href{https://github.com/juliansester/MDP_Bound}{https://github.com/juliansester/MDP$\_$Bound}.
\begin{figure}[htb!]
\includegraphics[scale=0.5]{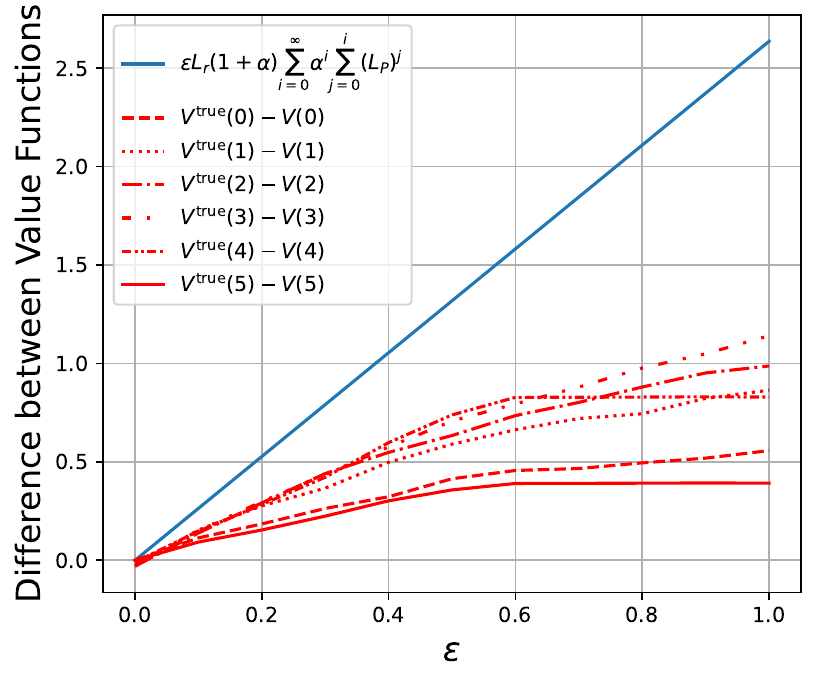}
\label{fig_coin_toss}
\caption{The difference between the non-robust and the robust value function compared with the upper bound from \eqref{eq_bound_main_thm_1} in the setting described in Example~\ref{exa_toin_coss} in dependence of $\varepsilon>0$ and for different initial values of the MDP. Initial values larger than $5$ are omitted due to the setting-specific symmetry $V(x_0)-V^{\rm true}(x_0) = V(10-x_0)-V^{\rm true}(10-x_0) $ for $x_0\in \{0,1,\dots,10\}$. }
\end{figure}

\end{exa}

\section{Proof of the main result}\label{sec_proofs}
In Section~\ref{sec_auxiliary} we provide several auxiliary lemmas which are necessary to establish the proof of Theorem~\ref{thm_main_result} reported in Section~\ref{sec_proof_main}.
\subsection{Auxiliary Results}\label{sec_auxiliary}
\begin{lem}\label{lem_1}
 Let $r:\X \times A \times \X \rightarrow \R$ satisfy  Assumption~\ref{asu_2}. Let  $\X \times A \ni (x,a) \mapsto \PP^{\operatorname{true}}(x,a)(\D x_1) \in \left(\mathcal{M}_1^q(\X), \tau_q\right)$ satisfy Assumption~\ref{asu_3}. For any\footnote{We denote here and in the following by $C_b(\X, \R)$ the set of continuous and bounded functions from $\X$ to $\R$.} $v \in C_b(\X, \R)$ define 
\begin{equation}\label{eq_defn_T_true}
\mathcal{T}^{\operatorname{true}}{v}(x_0):= \sup_{a \in A}{\int_{\X}}\left(r(x_0,a,{x_1})+\alpha v({x_1})\right) {\PP^{\operatorname{true}}(x_0,a)(\D x_1)},\qquad x_0 \in \X. 
\end{equation}
Then, for any $v \in C_b(\X, \R)$ being $L_r$-Lipschitz, $n \in \N$, $x_0,x_0'\in \X$, we have 
\begin{equation}\label{eq_lem_1}
\bigg|\left(\mathcal{T}^{\operatorname{true}}\right)^nv(x_0)-\left(\mathcal{T}^{\operatorname{true}}\right)^nv(x_0')\bigg| \leq L_r \left(1+L_P(1+\alpha) \sum_{i=0}^{n-1} \alpha^i L_P^i\right)\|x_0-x_0'\|.
\end{equation}
\end{lem}
\begin{proof}
For any $x_0,x_0'\in \X$, $a \in A$ let ${\Pi^{\operatorname{true}}}(\D {x_1}, \D {x_1'}) \in \mathcal{M}_1(\X \times \X)$ denote an optimal coupling between $\PP^{\operatorname{true}}(x_0,a)$ and $\PP^{\operatorname{true}}(x_0',a)$ w.r.t.\,$d_{W_1}$, i.e., 
\begin{equation}\label{eq_wasserstein_L_P}
\begin{aligned}
\int_{\X\times \X} \big\|{x_1}-{x_1'}\big\|~{\Pi^{\operatorname{true}}} (\D {x_1},\D {x_1'})&=d_{W_1}(\PP^{\operatorname{true}}(x_0,a),\PP^{\operatorname{true}}(x_0',a))\\
&\leq d_{W_q}(\PP^{\operatorname{true}}(x_0,a),\PP^{\operatorname{true}}(x_0',a)),
\end{aligned}
\end{equation}
where the inequality follows from Hölders's inequality, cf., e.g., \cite[Remark 6.6]{villani2008optimal}.
We prove the claim by induction. We start with the base case $n=1$, and compute by using the Lipschitz continuity of the functions $r$ and $v$ and of $\PP^{\operatorname{true}}$ that
\begin{align*}
&\bigg|\left(\mathcal{T}^{\operatorname{true}}\right)v(x_0)-\left(\mathcal{T}^{\operatorname{true}}\right)v(x_0')\bigg|\\
&=\bigg|\sup_{a \in A}{\int_{\X}\left(r(x_0,a,x_1)+\alpha v(x_1)\right) \PP^{\operatorname{true}}(x_0,a)(\D x_1)}-\sup_{a \in A}{\int_{\X}\left(r(x_0',a,x_1')+\alpha v(x_1')\right) \PP^{\operatorname{true}}(x_0',a)(\D x_1')}\\
&\leq \sup_{a \in A}\int_{\X\times \X} \bigg|r(x_0,a,{x_1})+\alpha v({x_1})-r(x_0',a,{x_1'})-\alpha v({x_1'})\bigg|~{\Pi^{\operatorname{true}}} (\D {x_1}, \D {x_1'})\\
&\leq L_r \|x_0-x_0'\|+L_r(1+\alpha)\sup_{a \in A}\int_{\X\times \X} \big\|{x_1}-{x_1'}\big\|~{\Pi^{\operatorname{true}}} (\D {x_1}, \D {x_1'})\\
&\leq L_r\|x_0-x_0'\|+L_r(1+\alpha)\sup_{a \in A}d_{W_q}\left(\PP^{\operatorname{true}}(x_0,a),\PP^{\operatorname{true}}(x_0',a)\right)\\
&\leq L_r \|x_0-x_0'\|+L_r(1+\alpha)L_P\|x_0-x_0'\|\\
&= L_r\big(1+(1+\alpha)L_P\big)\|x_0-x_0'\|.
\end{align*}
We continue with the induction step. Hence, let $n \in\N\cap[2,\infty)$ be arbitrary and assume that \eqref{eq_lem_1} holds for $n-1$. Then, we compute
\begin{equation}
\begin{aligned}\label{eq_Tn_diff_1}
&\bigg|\left(\mathcal{T}^{\operatorname{true}}\right)^nv(x_0)-\left(\mathcal{T}^{\operatorname{true}}\right)^nv(x_0')\bigg| \\
&\leq \sup_{a \in A}\int_{\X \times \X} \bigg|r(x_0,a,{x_1})+\alpha ({\T^{\operatorname{true}}})^{n-1}v({x_1})\\
&\hspace{4cm}-r(x_0',a,{x_1'})-\alpha ({\T^{\operatorname{true}}})^{n-1}v({x_1'})\bigg|~{\Pi^{\operatorname{true}}} (\D {x_1}, \D {x_1'})\\
&\leq L_r\|x_0-x_0'\|+L_r\sup_{a \in A}\int_{\X\times\X} \big\|{x_1}-{x_1'}\big\|~{\Pi^{\operatorname{true}}} (\D {x_1}, \D {x_1'})\\
&\hspace{2cm}+\alpha \sup_{a \in A}\int_{\X\times\X} \big|({\T^{\operatorname{true}}})^{n-1}v({x_1})-({\T^{\operatorname{true}}})^{n-1}v({x_1'})\big|~{\Pi^{\operatorname{true}}} (\D {x_1}, \D {x_1'}).
\end{aligned}
\end{equation}
Applying the induction hypothesis to \eqref{eq_Tn_diff_1} therefore yields
\begin{equation*}
\begin{aligned}\label{eq_Tn_diff_2}
&\bigg|\left(\mathcal{T}^{\operatorname{true}}\right)^nv(x_0)-\left(\mathcal{T}^{\operatorname{true}}\right)^nv(x_0')\bigg| \\
&\leq L_r \|x_0-x_0'\|+L_r\sup_{a \in A}\int_{\X\times\X} \big\|{x_1}-{x_1'}\big\|~{\Pi^{\operatorname{true}}} (\D {x_1}, \D {x_1'})\\
&\hspace{2cm}+\alpha L_r\left(1+L_P(1+\alpha)\sum_{i=0}^{n-2}\alpha^i L_P^i \right)\sup_{a \in A}\int_{\X\times\X} \big\|{x_1}-{x_1'}\big\|~{\Pi^{\operatorname{true}}} (\D {x_1}, \D {x_1'})\\
&\leq L_r \|x_0-x_0'\|+L_r\cdot L_P \|x_0-x_0'\|+\alpha L_r\left(1+L_P(1+\alpha)\sum_{i=0}^{n-2}\alpha^i L_P^i \right)L_P  \|x_0-x_0'\|\\
&=L_r \left(1+ (1+\alpha)L_P +L_P(1+\alpha)\sum_{i=0}^{n-2}\alpha^{i+1} L_P^{i+1} \right)  \|x_0-x_0'\|\\
&=L_r \left(1+L_P(1+\alpha)\sum_{i=0}^{n-1}\alpha^i L_P^{i} \right)  \|x_0-x_0'\|.
\end{aligned}
\end{equation*}
\end{proof}
\begin{lem}\label{lem_2}
Let Assumption~\ref{asu_2} 
and Assumption~\ref{asu_3} hold true.  Moreover, let
\begin{align*}
\X \times A \ni (x,a) &\mapsto \PP^{\operatorname{wc}}(x,a)  \in \mathcal{P}(x,a)
\end{align*}
denote another probability kernel contained in $\mathcal{P}(x,a)$ for each $x,a, \in \X \times A$. Furthermore, for any $v \in C_b(\X,\R)$ define 
\begin{equation}\label{eq_defn_T_wc}
 \mathcal{T}^{\operatorname{wc}}v(x_0):= \sup_{a\in A}{\int_{\X}}\left(r(x_0,a,{x_1})+\alpha v({x_1})\right) {\PP^{\operatorname{wc}}(x_0,a)(\D x_1)}, \qquad x_0 \in \X.
\end{equation}
\begin{itemize}
\item[(i)]
Then, for any $v \in C_b(\X, \R)$ being $L_r$-Lipschitz, $n \in \N$, $x_0\in \X$, we have
\begin{equation}\label{eq_lem2}
\bigg|\left( \mathcal{T}^{\operatorname{wc}}\right)^nv(x_0)-\left( \mathcal{T}^{\operatorname{true}}\right)^nv(x_0)\bigg|\leq 2 L_r \varepsilon \left(1+\alpha\right) \sum_{i=0}^{n-1}\alpha^i \sum_{j=0}^i (L_P)^j,
\end{equation}
where $\mathcal{T}^{\operatorname{true}}$ is defined in \eqref{eq_defn_T_true}. 
\item[(ii)]
Moreover,  in the special case that $\PP^{\operatorname{true}} = \widehat{\PP}$, we obtain for any $x_0\in \X$ that
\begin{equation}\label{eq_lem2_true}
\bigg|\left( \mathcal{T}^{\operatorname{wc}}\right)^nv(x_0)-\left( \mathcal{T}^{\operatorname{true}}\right)^nv(x_0)\bigg|\leq L_r \varepsilon \left(1+\alpha\right) \sum_{i=0}^{n-1}\alpha^i \sum_{j=0}^i (L_P)^j.
\end{equation}
\end{itemize}
\end{lem}
\begin{proof}
\begin{itemize}
\item[(i)]
For any $x_0\in \X$, $a \in A$, let ${\Pi}( \D {x_1},\D {x_1'}) \in \mathcal{M}_1(\X \times \X)$ denote an optimal coupling between  and $\PP^{\operatorname{wc}}(x_0,a)$ and $\PP^{\operatorname{true}}(x_0,a)$ w.r.t.\,$d_{W_1}$. Then, since both $\PP^{\operatorname{wc}}(x_0,a),\PP^{\operatorname{true}}(x_0,a) \in \mathcal{B}^{(q)}_\varepsilon\left(\widehat{\PP}(x_0,a)\right)$ we have
\begin{equation}\label{eq_wasserstein_2eps}
\begin{aligned}
\int_{\X\times \X} \big\|{x_1}-{x_1'}\big\|{\Pi} ( \D {x_1},\D {x_1'})&=d_{W_1}(\PP^{\operatorname{wc}}(x_0,a),\PP^{\operatorname{true}}(x_0,a))\\
&\leq d_{W_q}(\PP^{\operatorname{wc}}(x_0,a),\PP^{\operatorname{true}}(x_0,a))\leq 2 \varepsilon,
\end{aligned}
\end{equation}
where the first inequality follows from Hölders's inequality, cf., e.g., \cite[Remark 6.6]{villani2008optimal}.
We prove the claim by induction. To this end, we start with the base case $n=1$, and compute by using \eqref{eq_wasserstein_2eps} and the Lipschitz continuity of $r$, $v$, and of $\PP^{\operatorname{true}}$  that
\begin{align*}
&\bigg|\left( \mathcal{T}^{\operatorname{wc}}\right)v(x_0)-\left( \mathcal{T}^{\operatorname{true}}\right)v(x_0)\bigg| \\
&= \left|\sup_{a\in A} {\int_{\X}}\left(r(x_0,a,{x_1})+\alpha v({x_1})\right) {\PP^{\operatorname{wc}}(x_0,a)(\D x_1)}-\sup_{a\in A} {\int_{\X}}\left(r(x_0,a,{x_1'})+\alpha v({x_1'})\right) {\PP^{\operatorname{true}}(x_0,a)(\D x_1')}\right|\\
&\leq \sup_{a\in A}  \int_{\X\times\X} \big|r(x_0,a,{x_1})+\alpha v({x_1})-r(x_0,a,{x_1'})-\alpha v({x_1'})\big| ~{\Pi}  (\D {x_1},\D {x_1'})\\
&\leq L_r(1+\alpha)\sup_{a\in A}  \int_{\X\times\X} \big\|{x_1}-{x_1'}\big\| ~{\Pi} (\D {x_1},\D {x_1'})\\
&\leq L_r(1+\alpha)\sup_{a\in A}  d_{W_q}(\PP^{\operatorname{wc}}(x_0,a),\PP^{\operatorname{true}}(x_0,a)) \leq L_r(1+\alpha) \cdot 2\varepsilon.
\end{align*}
We continue with the induction step. Therefore, let $n \in\N\cap [2,\infty)$ be arbitrary and assume that \eqref{eq_lem2} holds for $n-1$. Then, we compute
\begin{align}
&\bigg|\left( \mathcal{T}^{\operatorname{wc}}\right)^nv(x_0)-\left( \mathcal{T}^{\operatorname{true}}\right)^nv(x_0)\bigg| \notag \\
&\leq \sup_{a\in A}  \int_{\X\times\X} \bigg|r(x_0,a,{x_1})+\alpha \left( \mathcal{T}^{\operatorname{wc}}\right)^{n-1}v({x_1})\\
&\hspace{3cm}-r(x_0,a,{x_1'})-\alpha  \left(\mathcal{T}^{\operatorname{true}}\right)^{n-1} v({x_1'})\bigg| ~{\Pi} (\D {x_1},\D {x_1'}) \notag \\
&\leq \sup_{a\in A}  \int_{\X\times\X} \big|r(x_0,a,{x_1})-r(x_0,a,{x_1'})
\big| ~{\Pi} (\D {x_1},\D {x_1'})
\notag \\
&\hspace{1.5cm}+\alpha \sup_{a\in A}  \int_{\X\times\X} \big|  \left( \mathcal{T}^{\operatorname{true}}\right)^{n-1}v({x_1})-\left(\mathcal{T}^{\operatorname{true}}\right)^{n-1} v({x_1'})\big| ~{\Pi} (\D {x_1},\D {x_1'})\label{eq_lem43_1}\\
&\hspace{1.5cm}+\alpha \sup_{a\in A}  \int_{\X\times\X} \big| \left(\mathcal{T}^{\operatorname{wc}}\right)^{n-1}v({x_1})- \left(\T^{\operatorname{true}}\right)^{n-1}v({x_1})\big| ~{\Pi} (\D {x_1},\D {x_1'}). \label{eq_lem43_2}
\end{align}
Applying Lemma~\ref{lem_1} to \eqref{eq_lem43_1} and the induction hypothesis to \eqref{eq_lem43_2} together with \eqref{eq_wasserstein_2eps} therefore yields
\begin{align*}
&\bigg|\left( \mathcal{T}^{\operatorname{wc}}\right)^nv(x_0)-\left( \mathcal{T}^{\operatorname{true}}\right)^nv(x_0)\bigg| \\
&\leq L_r \sup_{a \in A} \int_{\X\times\X} \big\| {x_1}- {x_1'}\big\|~ {\Pi} (\D {x_1},\D {x_1'})\\
&\hspace{1cm}+\alpha L_r\left(1+L_P(1+\alpha) \sum_{i=0}^{n-2} \alpha^i L_P^i\right) \int_{\X\times\X} \big\| {x_1}- {x_1'}\big\|~ {\Pi} (\D {x_1},\D {x_1'})\\
&\hspace{1cm}+\alpha \left(2L_r\varepsilon(1+\alpha) \sum_{i=0}^{n-2} \alpha^i \sum_{j=0}^i L_P^j\right)\\
&\leq L_r \cdot 2  \varepsilon +\alpha L_r\left(1+L_P(1+\alpha)\sum_{i=0}^{n-2} \alpha^iL_P^i\right)2\varepsilon +\alpha \left(L_r \cdot 2  \varepsilon(1+\alpha)\sum_{i=0}^{n-2} \alpha^i \sum_{j=0}^i L_P^j\right)\\
&= 2L_r  \varepsilon(1+\alpha)\left(1+\alpha L_P\sum_{i=0}^{n-2} \alpha^iL_P^i+\sum_{i=0}^{n-2} \alpha^{i+1} \sum_{j=0}^i L_P^j\right)\\
&= 2L_r \varepsilon(1+\alpha)\left(\sum_{i=0}^{n-1} \alpha^iL_P^i+\sum_{i=1}^{n-1} \alpha^{i} \sum_{j=0}^{i-1} L_P^j\right) = 2L_r \varepsilon(1+\alpha)\left(\sum_{i=0}^{n-1} \alpha^{i} \sum_{j=0}^{i}L_P^j\right).
\end{align*}
\item[(ii)]
In the case $\PP^{\operatorname{true}} = \widehat{\PP}$ we have for  any $x_0\in \X$, $a \in A$ that
\begin{equation}\label{eq_better_inequality}
d_{W_q}(\PP^{\operatorname{wc}}(x_0,a),\PP^{\operatorname{true}}(x_0,a))\leq \varepsilon,
\end{equation}
since the ambiguity set $\mathcal{P}(x_0,a)$ is centered around $\PP^{\operatorname{true}}(x_0,a)=\widehat{\PP}(x_0,a)$.
Hence, replacing \eqref{eq_wasserstein_2eps} by \eqref{eq_better_inequality} and then following the proof of (i) shows the assertion.

\end{itemize}
\end{proof}

\begin{lem}\label{lem_3} Let $0<\alpha<1$ and $L_P\geq 0$ satisfy $\alpha \cdot L_P <1$. Then 
\begin{equation}
\sum_{i=0}^\infty \alpha^i \sum_{j=0}^i (L_P)^j < \infty.
\end{equation}
\end{lem}
\begin{proof}
Note that 
\begin{equation}\label{eq_proof_lem_3_1}
0 \leq \sum_{i=0}^\infty \alpha^i \sum_{j=0}^i (L_P)^j \leq  \sum_{i=0}^\infty (i+1) \cdot \alpha^i \max\{1,L_P\}^i = :\sum_{i=0}^\infty a_i,
\end{equation}
with $a_i = (i+1) \cdot \alpha^i \max\{1,L_P\}^i $. Moreover
\[
\frac{a_{i+1}}{a_i} = \frac{(i+2) \cdot \alpha^{i+1} \max\{1,L_P\}^{i+1}}{(i+1) \cdot \alpha^i \max\{1,L_P\}^i}= \frac{i+2}{i+1} \cdot \alpha \cdot \max\{1,L_P\} \rightarrow \alpha \cdot \max\{1,L_P\} < 1 \text{ as } i \rightarrow \infty.
\]
Hence, d'Alembert's criterion implies that $\sum_{i=0}^{\cdot } a_i$ converges absolutely. Thus, by \eqref{eq_proof_lem_3_1}, we have $\sum_{i=0}^\infty \alpha^i \sum_{j=0}^i (L_P)^j <\infty$.
\end{proof}
\begin{lem}\label{le:DPP-holds}
	Let Assumptions~\ref{asu_2}--~\ref{asu_alpha} hold true. Then  $\mathcal{P}(x,a):=\mathcal{B}^{(q)}_\varepsilon\big(\widehat{\PP}(x,a)\big)$ defined in \eqref{eq_definition_wasserstein_ball} satisfies \cite[Standing Assumption~2.2]{neufeld2023markov} and the reward function $r:\X \times A \times \X \rightarrow \R$ together with the discount factor $0<\alpha<1$ satisfy  \cite[Standing Assumption~2.4]{neufeld2023markov}.
	\\
	As a consequence,  \cite[Theorem~2.7]{neufeld2023markov} then directly implies that the dynamic programming principle holds for the robust Markov decision problem defined in \eqref{eq_robust_problem_1}.
\end{lem}
\begin{proof}
First, if $r:\X \times A \times \X \rightarrow \R$ is bounded, then Assumptions~\ref{asu_2}--~\ref{asu_alpha} allows us to use  \cite[Proposition~3.1]{neufeld2023markov} which immediately ensures that the result holds true with respect to $p=0$ and $C_P=1$ in the notation of \cite[Standing Assumptions~2.2\&~2.4]{neufeld2023markov}.

Now, assume for the rest of this proof that $r:\X \times A \times \X \rightarrow \R$ is unbounded.
Then by Assumption~\ref{asu_2}(ii) we have that $q\in[2,\infty)\cap\N$. 
In this case, let $p=1$ in the notation of \cite[Standing Assumptions~2.2~\&~2.4]{neufeld2023markov}. Then our Assumptions~\ref{asu_2}~\&~\ref{asu_alpha} immediately ensure that \cite[Standing Assumption~2.4]{neufeld2023markov} holds.
Moreover, by our Assumption~\ref{asu_p}, we directly obtain from \cite[Proposition 4.1]{neufeld2024non} that \cite[Standing Assumption~2.2(i)]{neufeld2023markov} holds. Therefore, it remains to verify \cite[Standing Assumptions~2.2(ii)]{neufeld2023markov}. To that end, let
\begin{equation}\label{C_P}
C_P:= \max\Bigg\{	1+	\varepsilon
+
\sup_{a\in A}\inf_{x'\in \X}\bigg\{
\int_{\X }\|z\|\, \widehat{\PP}(x',a) (\D z) 
+
L_{\widehat{\PP}}\|x'\|
\bigg\}, L_{\widehat{\PP}}\Bigg\}<\infty.
\end{equation}
Indeed note that $C_P<\infty$, as Assumption~\ref{asu_p} ensures that
the map 
\begin{equation*}
	\X \times A\ni (x',a)\mapsto \int_{\X }\|z\|\,\widehat{\PP}(x',a)  (\D z)  +
	L_{\widehat{\PP}}\|x'\| \in [0,\infty)
\end{equation*} 
is continuous. This  implies that the map 
\begin{equation*}
	A\ni a\mapsto \inf_{x'\in \X}\bigg\{\int_{\X }\|z\|\, \widehat{\PP}(x',a) (\D z)+
	L_{\widehat{\PP}}\|x'\|\bigg\}\in [0,\infty)
\end{equation*} is upper semicontinuous,  which in turns ensures that $C_P$ is finite as $A$ is compact.
Now, let $(x,a)\in \X \times A$ and $\PP\in \mathcal{P}(x,a)=\mathcal{B}^{(q)}_\varepsilon\big(\widehat{\PP}(x,a)\big)$ be arbitrarily chosen. Then by following the calculations in \cite[Proof of Proposition 4.1, Equation (6.34)]{neufeld2024non} (with $p=1$ in the notation of  \cite{neufeld2024non}) using the Lipschitz continuity of $\widehat{\PP}$ we obtain for any arbitrary $x'\in \X$ that
\begin{equation*}
	\begin{split}
\int_{\X} 1+\|y\|\,\PP( \D y)
	&\leq
1+	\varepsilon
	+
\int_{\X }\|z\|\, \widehat{\PP}(x',a) (\D z) 
	+
	L_{\widehat{\PP}}(\|x'\|	+ \|x\|).
\end{split}
\end{equation*}
Since $x'\in \X$ was arbitrarily chosen, we see from \eqref{C_P} that
\begin{equation*}
	\begin{split}
	\int_{\X} 1+\|y\|\,\PP( \D y)
		&\leq
		1+	\varepsilon
		+
		\sup_{a\in A}\inf_{x'\in \X}\bigg\{
		\int_{\X }\|z\| \,\widehat{\PP}(x',a) (\D z) 
		+
			L_{\widehat{\PP}}\|x'\|
			\bigg\}
			+
		L_{\widehat{\PP}} \|x\|\\
		& \leq C_P(1+ \|x\|),
	\end{split}
\end{equation*}
which shows that \cite[Standing Assumption~2.2(ii)]{neufeld2023markov} indeed holds. 
\end{proof}
\subsection{Proof of Theorem ~\ref{thm_main_result}}\label{sec_proof_main}
\begin{itemize}
\item[(i)]
First note that as by assumption $\PP^{\operatorname{true}}(x,a) \in \mathcal{P}(x,a)$ for all $(x,a) \in \X \times A$, we have 
\[
0 \leq V^{\operatorname{true}}(x_0)-V(x_0) \text{ for all } x_0 \in \X.
\]
To compute the upper bound, we fix any $v \in C_b(\X, \R)$ which is $L_r$-Lipschitz and we define  the operator $\mathcal{T}^{\operatorname{true}}$ by \eqref{eq_defn_T_true}.
Then, by Lemma~\ref{le:DPP-holds} and \cite[Theorem 2.7~(ii)]{neufeld2023markov}, we have
\begin{equation}\label{eq_proof_thm_eq_3}
V^{\operatorname{true}}(x_0)= \lim_{n \rightarrow \infty} \left(\mathcal{T}^{\operatorname{true}}\right)^nv(x_0), \qquad \text{  } V(x_0)= \lim_{n \rightarrow \infty} \left(\mathcal{T}\right)^nv(x_0)
\end{equation}
for all $x_0 \in \X$ and for $\mathcal{T}$ as defined in \cite[Equation (8)]{neufeld2023markov}. Moreover, by \cite[Theorem 2.7~(iii)]{neufeld2023markov}, there exists a \emph{worst case} transition kernel $\X \times A \ni (x,a)\mapsto \PP^{\operatorname{wc}}(x,a) $ with $\PP^{\operatorname{wc}}(x,a)  \in \mathcal{P}(x,a)$ for all $(x,a) \in \X \times A$ such that, by denoting for any $\ab =(a_t)_{t\in \N_0}\in \mathcal{A}$
 $$\PP_{x_0,\ab}^{\operatorname{wc}}:=\delta_{x_0} \otimes  \PP^{\operatorname{wc}} \otimes  \PP^{\operatorname{wc}} \otimes  \PP^{\operatorname{wc}} \otimes  \PP^{\operatorname{wc}} \cdots \in \mathcal{M}_1(\Omega),$$
we have 
\begin{equation}\label{eq_proof_thm_eq_4}
V(x_0) = \sup_{\ab \in \mathcal{A}}\E_{\PP_{x_0,\ab}^{\operatorname{wc}}}\left[\sum_{t=0}^\infty \alpha^t r(X_t,a_t(X_t),X_{t+1})\right] =\lim_{n \rightarrow \infty} \left(\mathcal{T}^{\operatorname{wc}}\right)^nv(x_0),\qquad x_0 \in \X,
\end{equation}
where $\mathcal{T}^{\operatorname{wc}}$ is defined in \eqref{eq_defn_T_wc}.
Therefore by \eqref{eq_proof_thm_eq_3}, \eqref{eq_proof_thm_eq_4}, Lemma~\ref{lem_2}, and Lemma~\ref{lem_3}, we have for all $x_0 \in \X$ that
\begin{align}
V^{\operatorname{true}}(x_0)-V(x_0) &= \lim_{n \rightarrow \infty} \left(\mathcal{T}^{\operatorname{true}}\right)^nv(x_0)-\lim_{n \rightarrow \infty} \left(\mathcal{T}^{\operatorname{wc}}\right)^nv(x_0)  \notag
\\&\leq \lim_{n \rightarrow \infty} \bigg|\left(\mathcal{T}^{\operatorname{true}}\right)^nv(x_0)-\left(\mathcal{T}^{\operatorname{wc}}\right)^nv(x_0)\bigg| \notag \\
&\leq 2 L_r \varepsilon \left(1+\alpha\right)\lim_{n \rightarrow \infty}  \sum_{i=0}^{n-1} \alpha^i \sum_{j=0}^i L_P^j =  2 L_r \varepsilon \left(1+\alpha\right)\sum_{i=0}^\infty \alpha^i \sum_{j=0}^i L_P^j  < \infty. \label{eq_thm_proof_lastline}
\end{align}
\item[(ii)]
In the case $\PP^{\operatorname{true}} = \widehat{\PP}$, due to Lemma~\ref{lem_2}~(ii), we may use \eqref{eq_lem2_true} and replace inequality \eqref{eq_thm_proof_lastline}
by 
\[
V^{\operatorname{true}}(x_0)-V(x_0) \leq   L_r \varepsilon \left(1+\alpha\right)\lim_{n \rightarrow \infty}  \sum_{i=0}^{n-1} \alpha^i \sum_{j=0}^i L_P^j =   L_r \varepsilon \left(1+\alpha\right)\sum_{i=0}^\infty \alpha^i \sum_{j=0}^i L_P^j  < \infty. 
\]
\end{itemize}
\vspace{0.3cm}

\section*{Acknowledgments}
\noindent
We thank an anonymous referee of the paper \cite{neufeld2022robust} who raised a question that led to this note.\\
Moreover, financial support by the MOE AcRF Tier 1 Grant \emph{RG74/21} and by the  Nanyang Assistant Professorship Grant (NAP Grant) \emph{Machine Learning based Algorithms in Finance and Insurance} is gratefully acknowledged. 

\bibliographystyle{plain} 
\bibliography{literature}
\end{document}